\documentclass[12pt]{article}

\usepackage[section]{placeins}
\usepackage{natbib}
\usepackage[letterpaper, margin=1in]{geometry}
\usepackage{breakcites}

\usepackage[todo,table]{MathEnv}
\usepackage[stat,commdiag,alg]{MathShorthand}

\newcommand{\new}[1]{{#1}}

\title{The Topological Behavior of Preferential Attachment Graphs}

\author{Chunyin Siu}
\date{}

\begin{document}

\maketitle
\abstract{
We investigate the higher-order connectivity of scale-free networks using algebraic topology. We model scale-free networks as preferential attachment graphs, and we study the algebraic-topological properties of their clique complexes. We focus on the Betti numbers and the homotopy-connectedness of these complexes. We determine the asymptotic almost sure orders of magnitude of the Betti numbers. We also establish the occurence of homotopical phase transitions for the infinite complexes, and we determine the critical thresholds at which the homotopy-connectivity changes. This partially verifies Weinberger's conjecture on the homotopy type of the infinite complexes. We conjecture that the mean-normalized Betti numbers converge to power-law distributions, and we present numerical evidence. Our results also highlight the subtlety of the scaling limit of topology, which arises from the tension between topological operations and analytical limiting process. We discuss such tension at the end of the Introduction.
}




\section{Introduction}

Many real-world networks are believed to be scale-free, in the sense that their degree distributions obey power laws, whose variance is often infinite \citep{barabasi99_preferentialAttachment}. In particular, preferential attachment graphs, defined in \cref{def:preferential_attachment}, are popular models for such networks. In these graphs, nodes are inductively added and attached to $m$ randomly chosen previous nodes. At each discrete time step, the new node is more likely to attach to high-degree nodes. The extent of this likeliness, called the strength of preferential attachment, can be controlled by a real parameter, often denoted by $\delta$. The graph-theoretical properties of scale-free networks and preferential attachment graphs have been extensively studied, and we refer the reader to \citep{hofstad16_randomGraphs1,hofstad24_randomGraphs2} for comprehensive surveys.

Recently, there has been substantial interest in the higher-order connectivity in various (not necessarily scale-free) networks \citep{benson16_complexNetwork_motif,watts98_clusteringCoefficient,nolte20_brain_higherOrderNetwork,lambiotte19_higherOrderNetworks,battiston20_higherOrderNetworks,bianconi21_higherOrderNetwork}. In this work, we study higher-order connectivity using concepts from algebraic topology.

For instance, \emph{homotopy-connectedness} generalizes path-connectedness.

\begin{definition}[Homotopy and Homotopy-Connectedness]\label{def:homotopy_connected} $\quad$
\begin{itemize}
\item Two maps $f, g: X \to Y$ such that $f(x_0) = g(x_0) = y_0$ are said to be homotopic if there exists a map $F: X \times [0, 1] \to X$ such that $F(\cdot, 0) = f$, $F(\cdot, 1) = g$, and $F(x_0, \cdot) = y_0$ (Cf. p.3 of \citep{hatcher02_algtopo}).
\item For each nonnegative integer $q$, let $S^q$ be the $q$-dimensional sphere and fix an $s_q \in S^q$. Let $X$ be a nonempty topological space. $X$ is said to be $q$-homotopy-connected if for every $x_0 \in X$ and every integer $0 \leq r \leq q$, every map $f: S^r \to X$ such that $f(s_r) = x_0$ is homotopic to the constant map $s \mapsto x_0$ (Cf. p.346 of \citep{hatcher02_algtopo}).
\end{itemize}
\end{definition}

$0$-homotopy-connectedness is equivalent to path-connectedness ($S^0 = \{-1, 1\} \subseteq \mathbb{R}$). In general, $q$-homotopy-connected spaces are more connected when $q$ is larger, in the sense that $q$-homotopy-connected spaces are automatically $r$-homotopy-connected if $0 \leq r < q$. Intuitively, a $q$-homotopy-connected space has no spherical holes of dimension less than or equal to $q$. We further discuss homotopy-connectedness in \cref{sec:homotopy_theory}, in particular, after \cref{def:homotopy_equivalent_via_group}.

Despite the rich theory of algebraic topology, much less is known about the algebraic-topological properties of scale-free networks. \new{Most variants of preferential attachment graphs are path-connected by construction, but it is unclear whether they (more precisely, their associated complexes) are homotopy-connected.}

\subsection{Homological and Homotopical Properties of Preferential Attachment}

\new{
Between the two main branches of algebraic topology, namely homology theory and homotopy theory, computations in homotopy theory are generally more difficult. Indeed, even the homotopy properties of \emph{spheres} are not completely understood. We explain this further at the end of \cref{sec:homotopy_theory}. This difficulty partly explains the gap in the literature about the homotopy-connectedness of scale-free networks.

In terms of the homological properties of scale-free networks, 
}
the orders of magnitude of the mean \emph{Betti numbers} of finite preferential attachment \emph{clique complexes} were determined in \citep{siu23_PrefAtt_betti}, and critical thresholds at which the mean Betti numbers start exhibiting unbounded growth were discovered there as well.

Intuitively, the $q$-dimensional Betti number of a topological space is the number of independent $q$-dimensional holes (possibly non-spherical) in the space. Its formal definition is given below. The definitions of simplicial complexes and clique complexes are given in \cref{sec:homology_theory}. There we also further discuss their Betti numbers, and we contrast this homological property with homotopical properties in \cref{sec:homology_vs_homotopy}.

\begin{definition}[Homology Group, Betti Number, Cycle and Boundary; Section 5 of \citep{munkres84algtopo}]
\label{def:homology}
Let $X$ be a simplicial complex with totally ordered vertices.
\begin{itemize}
\item For each nonnegative integer $q$, let $C_q(X)$ be the free abelian group generated by the $q$-dimensional simplices of $X$, and let $\partial_q: C_q \to C_{q-1}$ be the homomorphism defined by
$$\partial_q\{x_0, ..., x_q\} = \sum_{0 \leq i \leq q} (-1)^i \{x_0, ..., \hat x_i, ..., x_q\}$$
whenever $\{x_0, ..., x_q\}$ is a simplex of $X$ and $x_0 < ... < x_q$,
where the hat denotes removal, e.g. $\{x_0, \hat x_1, x_2\} = \{x_0, x_2\}$.
\item The homology group $H_q(X)$ and the Betti number $\beta_q(X)$ of $X$ at dimension $q$ are defined by 
\begin{align*}
H_q(X) &= \ker \partial_q / \im \partial_{q+1}
\\
\beta_q(X) &= \rk H_q(X),
\end{align*}
where $/$ denotes group quotient.
\item Elements of $H_q(X)$, $\ker \partial_q$ and $\im \partial_{q+1}$ are called, respectively, homology classes, cycles and boundaries of dimension $q$.
\end{itemize}
\end{definition}

\begin{remark}$\quad$
\begin{itemize}
\item The quotient in the definition of the homology group $H_q(X)$ is well-defined because one can show that $\partial_q \partial_{q+1} = 0$, and hence $\im \partial_{q+1} \subseteq \ker \partial_q$.
\item Homology classes are represented by cycles. Boundaries represent the trivial homology class, and they do not contribute towards the Betti numbers.
\end{itemize}
\end{remark}

\subsection{Highlights on our Main Results}
\label{sec:highlights}

\begin{figure}
\centering
\includegraphics[height = 4cm]{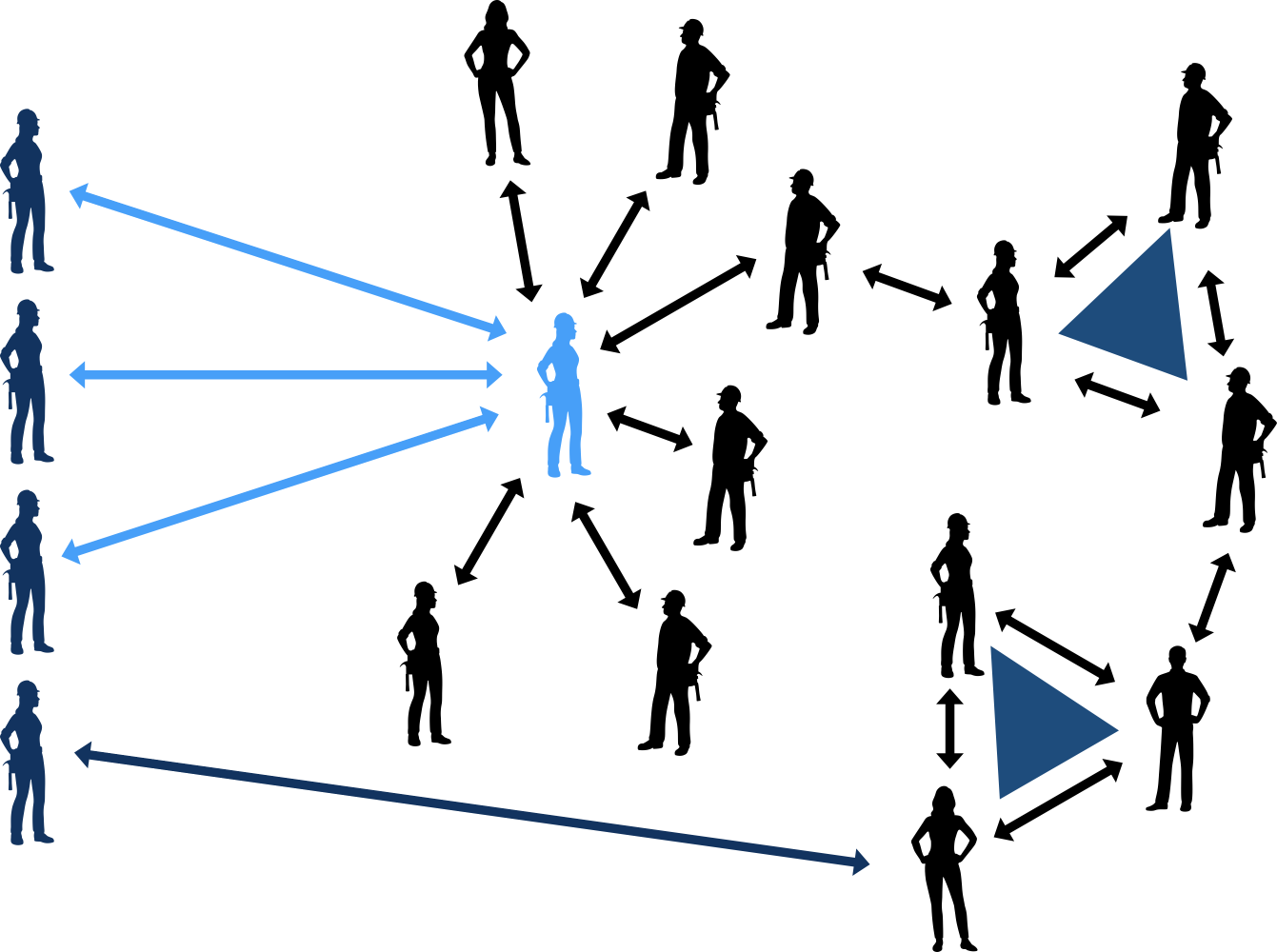}
\includegraphics[height = 4cm]{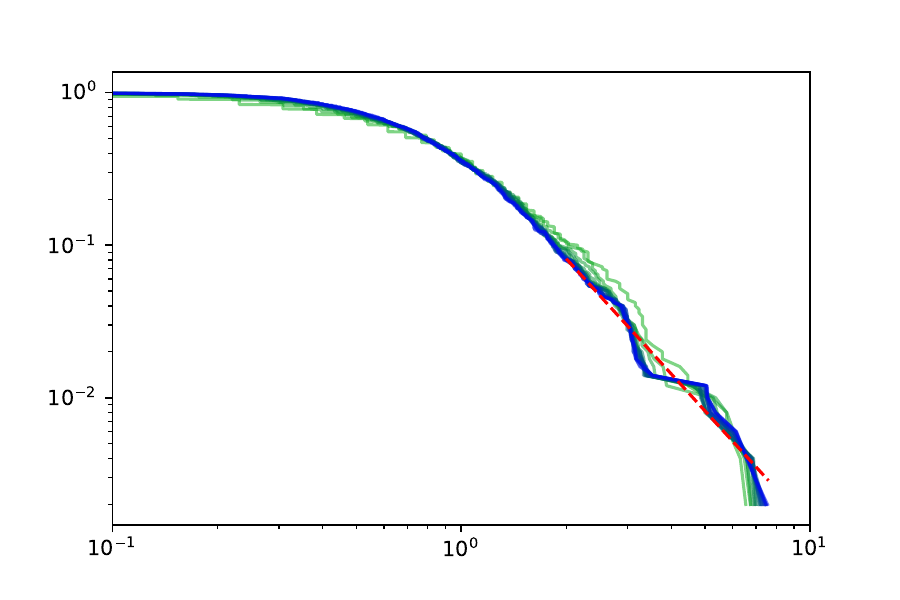}
\caption{(Left; Figure 1 of \citep{siu23_PrefAtt_betti}) An illustration of the preferential attachment and clique building mechanism. When new nodes (drawn as people) in the left column are added to the network, they are more likely to attach to already popular nodes (who have high degrees), like the light blue person in the figure. Fully connected subsets of nodes form triangles, tetrahedra and their higher-dimensional analogues in the clique complex. Note that in order to have triangles, each new node must connect to at least 2 nodes, but we only drawn one connection for each new node to keep the illustration simple. See \cref{sec:setup} for the precise definitions.
(Right) The log-log plot of the evolution of the complementary cumulative distribution functions ($\log (1-F(w))$ against $\log w$) of the mean-normalized Betti numbers at dimension 2. Green curves correspond to the distributions for complexes with fewer nodes, and blue ones, larger complexes. The dotted red line is the line of best-fit for the largest complex. Its slope is $-2.51$. Model parameters are detailed in \cref{sec:simulations}.}
\label{fig:teaser}
\end{figure}

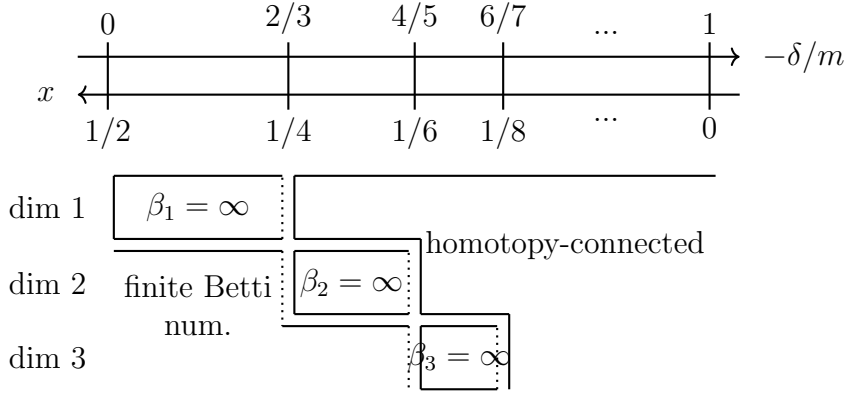
\begin{figure}
\centering
\begin{tikzpicture}[scale = 8, thick]
\draw[->,](-1.05, 0.188) -- (0.05, 0.188) node [label = right: $-\delta/m$] {};
\draw[->,](0.05, 0.125) -- (-1.05, 0.125) node [label = left: $x$] {};

\draw[](-1, 0.212) -- (-1, 0.1);
\node[label = above: 0] at (-1, 0.188) {};
\node[label = below: 1/2] at (-1, 0.125) {};
\draw[](-0.7, 0.212) -- (-0.7, 0.1);
\node[label = above: 2/3] at (-0.7, 0.188) {};
\node[label = below: 1/4] at (-0.7, 0.125) {};
\draw[](-0.49, 0.212) -- (-0.49, 0.1);
\node[label = above: 4/5] at (-0.49, 0.188) {};
\node[label = below: 1/6] at (-0.49, 0.125) {};
\draw[](-0.343, 0.212) -- (-0.343, 0.1);
\node[label = above: 6/7] at (-0.343, 0.188) {};
\node[label = below: 1/8] at (-0.343, 0.125) {};
\draw[](0, 0.212) -- (0, 0.1);
\node[label = above: 1] at (0, 0.188) {};
\node[label = below: 0] at (0, 0.125) {};

\node[label = above: ...] at (-0.172, 0.188) {};
\node[label = below: ...] at (-0.172, 0.125) {};

\draw[](-0.99, -0.115) -- (-0.71, -0.115);
\draw[](-0.99, -0.01) -- (-0.71, -0.01);
\draw[](-0.99, -0.115) -- (-0.99, -0.01);
\draw[dotted](-0.71, -0.115) -- (-0.71, -0.01);
\draw (-0.85, -0.0625) node {$\beta_1 = \infty$};
\draw[](-0.69, -0.24) -- (-0.5, -0.24);
\draw[](-0.69, -0.135) -- (-0.5, -0.135);
\draw[](-0.69, -0.24) -- (-0.69, -0.135);
\draw[dotted](-0.5, -0.24) -- (-0.5, -0.135);
\draw (-0.595, -0.188) node {$\beta_2 = \infty$};
\draw[](-0.48, -0.365) -- (-0.353, -0.365);
\draw[](-0.48, -0.26) -- (-0.353, -0.26);
\draw[](-0.48, -0.365) -- (-0.48, -0.26);
\draw[dotted](-0.353, -0.365) -- (-0.353, -0.26);
\draw (-0.416, -0.312) node {$\beta_3 = \infty$};

\draw (-1.1, -0.0625) node {dim 1};
\draw (-1.1, -0.188) node {dim 2};
\draw (-1.1, -0.312) node {dim 3};
\draw (-0.24, -0.125) node {homotopy-connected};
\draw (-0.85, -0.225) node[text width = 80, align = center] {finite Betti num.};

\draw[](-0.69, -0.01) -- (0.01, -0.01);
\draw[](-0.69, -0.115) -- (-0.48, -0.115);
\draw[](-0.48, -0.24) -- (-0.333, -0.24);
\draw[](-0.69, -0.01) -- (-0.69, -0.115);
\draw[](-0.48, -0.115) -- (-0.48, -0.24);
\draw[](-0.333, -0.24) -- (-0.333, -0.365);

\draw[](-0.99, -0.135) -- (-0.71, -0.135);
\draw[](-0.71, -0.26) -- (-0.5, -0.26);
\draw[dotted](-0.71, -0.135) -- (-0.71, -0.26);
\draw[dotted](-0.5, -0.26) -- (-0.5, -0.365);
\end{tikzpicture}
\caption{Phase transitions at different dimensions for infinite preferential attachment clique complexes for moderate $m$. The symbols $\delta$ and $m$ were introduced in the first paragraph of the Introduction and they are precisely defined in \cref{def:preferential_attachment}. The symbol $x$, defined in \cref{eqn:x}, is a monotone function of $\delta/m$. The strength of preferential attachment increases from left to right. The annotated conditions hold almost surely in their respective regions. Dotted lines indicate that the condition does not hold at the corresponding endpoint. See \cref{sec:main_results} for the precise statements.}
\label{fig:phase_diagram_infinite_complex}
\end{figure}

\new{
This work explores the limiting homological and homotopical properties of affine preferential attachment graphs.
The two pursuits are not completely independent, as homology and homotopy theories are intimately related.
}

\cref{thm:betti_ass_orderOfMagnitude} gives the asymptotic almost sure limit of the orders of magnitude of the Betti numbers of finite preferential attachment clique complexes.

We also study the topological properties of the clique complexes of \emph{infinite} preferential attachment graphs, where nodes are attached inductively by the preferential attachment mechanism \emph{ad infinitum}.

In \cref{thm:homotopy_connected}, \new{
we show that these infinite complexes are $q$-homotopy-connected if the strength of preferential attachment exceeds a $q$-dependent critical threshold. We establish our result by leveraging Barmak's sufficient condition for homotopy-connectedness in \citep{Farber23_largeRandomSimpComp}. To our knowledge, \cref{thm:homotopy_connected} is the first result on the homotopy properties of scale-free networks.

\cref{thm:infinite_betti} implies the threshold in \cref{thm:homotopy_connected} is tight. It says that the Betti numbers are infinite if the preferential attachment strength is slightly weaker than the threshold in \cref{thm:homotopy_connected}, and hence cannot be homotopy-connected, by Hurewicz's theorem (\cref{thm:hurewicz}).

These two theorems together give a negative answer to Weinberger's question in \citep{siu23_PrefAtt_betti} on the \emph{contractibility} of the infinite complexes, as Whitehead's theorem (\cref{thm:whitehead}) implies that a simplicial complex is contractible if and only if it is $q$-homotopy-connected for every $q \geq 0$. Contractibility is defined in \cref{def:contractible}.

Regarding the homological properties of the infinite complexes, these two theorems enrich our understanding of the higher-dimensional homological phase transitions for preferential attachment complexes, as they imply that the Betti numbers of the infinite complexes change from almost surely infinite to almost surely zero at the aforementioned critical thresholds. This \emph{almost sure} statement is stronger than the result in \citep{siu23_PrefAtt_betti} regarding the \emph{expected} Betti numbers. This phase transition is illustrated in \cref{fig:phase_diagram_infinite_complex}.

This shows that infinite complexes belong to a wide class of random simplicial complexes that undergo two phase transitions for each dimension, one at which many cycles emerge, and one at which most cycles become boundaries \citep{kahle14_randomCliqueComplex,kahle11_geometricComplex}. \new{The aforementioned drop of Betti numbers from infinity to zero is the second phase transition. \cref{thm:infinite_betti} also gives the critical threshold at which Betti numbers jump from almost surely finite to almost surely infinite. The two phase transitions demarcate the two endpoints of each box labeled $\beta_q = \infty$ in \cref{fig:phase_diagram_infinite_complex} for small $q$'s. This is in stark contrast with \emph{finite}} preferential attachment clique complexes, which do \emph{not} belong to this two-phase-transition class, because cycles appear much more frequently than boundaries in the finite complexes \citep{siu23_PrefAtt_betti}. However, this no longer applies for the infinite complexes, as the rates at which cycles appear and become boundaries are no longer relevant at infinity. Phase transitions of other random simplicial complexes are reviewed in \cref{sec:literature_review}.
}



\cref{thm:homotopy_connected} also affirms the observation in \citep{siu23_PrefAtt_betti} that preferential attachment favors higher-order connectivity. \cref{thm:homotopy_connected} shows that, the stronger the preferential attachment strength, the more connected the infinite complex is, in the sense that it is $q$-homotopy-connected for a larger $q$. Such connectivity arises from the fact that under strong preferential attachment effect, later nodes can fill up larger gaps among ancient nodes.


We conjecture that the Betti numbers $\beta_q^{(T)}$ of finite preferential attachment clique complexes with $T$ nodes admit a scaling limit, in the sense that the mean-normalized Betti numbers $\beta_q^{(T)}/E[\beta_q^{(T)}]$ converge in distribution as $T \to \infty$. The right panel of \cref{fig:teaser} shows the evolution of the complementary cumulative distribution functions ($1-\text{cdf}$) of the mean-normalized Betti numbers. Visually, the nearby curves suggest the distributions converge. Further details about this simulation, including model parameters and Kolmogorov-Smirnov statistics, will be detailed in \cref{sec:simulations} to support the conjectural convergence. The curves in the right panel of \cref{fig:teaser} are visually pretty straight, and the fitted slope of the limiting distribution is $-2.51$, which means the power-law distribution has a finite variance. We note that this exponent depends on the model parameters, and it is not necessarily larger than $2$. More simulation results for other model parameters than the ones for the right panel of \cref{fig:teaser} will be presented in \cref{sec:simulations}. Codes for our simulations are available at The GitHub repo \href{https://github.com/carolinerongyi/Preferential_Attachment_Clique_Complex}{\texttt{carolinerongyi/Preferential\_Attachment\_Clique\_Complex}}.

\subsection{Intuition for our Homotopy-Connectedness Result}

It was observed in \citep{siu23_PrefAtt_betti} that the growth in Betti numbers in preferential attachment clique complexes is driven by the formation of (not necessarily disjoint) small spheres, and these spheres become boundaries at a much slower rate than their formation. Shmuel Weinberger conjectured that, since all such spheres should eventually become boundaries of balls, the \emph{infinite} complexes, where the rate is no longer relevant, may be contractible. Our result formally verifies that such spheres are indeed the main obstruction to homotopy-connectedness. The infinite complexes are, however, not contractible, because the restriction on the number of edges implies certain cycles can never become boundaries.

\subsection{The Subtlety of the Limiting Topology}

The above intuition also explains the apparent paradox that Betti numbers of the finite complexes, by \cref{thm:betti_ass_orderOfMagnitude}, diverge to infinity, while the infinite complexes, by \cref{thm:homotopy_connected}, are homotopy-connected (and hence their Betti numbers vanish). For readers who are interested in \emph{persistent homology} (Cf. Chapter VII.1 of \citep{edelsbrunner10comptopo}), this discrepancy may be seen as a consequence of the facts that the persistence diagrams of preferential attachment clique complexes, filtered by node arrival, have infinitely many points with high probability, but almost surely none of them has an infinite death time.

This subtlety, however, does highlight the difficulty of studying the scaling limits of topological properties of random models. Topological operations may not commute with analytical limits, in the sense that a topological property of an analytically defined limiting object of a random model is not necessarily the same as the analytical limit of the corresponding topological properties of the random model. In symbols, we have
$$\lim \textbf{TopOp}(X_n) \neq \textbf{TopOp} (\lim X_n)$$
for many topological operations $\textbf{TopOp}$ and many (analytically defined) limiting operations $\lim$.

There are two main sources of noncommutativity:

\begin{description}
\item [The Global Nature of Topology] Since topological properties are often global in nature, locally defined analytical limits may fail to capture these properties.
\item [Small Support of Nontrivial Topological Behavior] \new{Many analytical limiting operations involve some kind of rescaling, e.g. the central limit theorem states that, upon scaling by $\sqrt{N}$, $\sum_{1 \leq i \leq N} X_i$ converges in distribution to a normal distribution under mild assumptions. In particular, $o(\sqrt{N})$ variations vanish in the limit upon rescaling.} If the nontrivial topological behavior is supported on a very small part of the spaces concerned, it may \new{vanish} upon rescaling when taking an analytical limit.
\end{description}

\new{We illustrate these sources with examples below, but before that,} we note that we have been carefully qualifying all instances of ``limit'' with ``analytical'', because some topological operations do commute with limiting procedures in \emph{category theory}, at least to a certain extent. For example, the theorem in Chapter 14.6 of \citep{may99_algtopo} 
states that, under mild conditions, the homology groups of the \emph{colimit} 
of a sequence of nested spaces are the 
colimit 
of their homology groups. 
However, for many random models, the underlying category is unclear, and so are the relevant morphims between different terms in the sequence. Even when these are clear, the categorical limit and the analytical limit may have very different properties.

Now, we give examples to illustrate the two sources of noncommutativity above, as well as the discrepancy between analytical and categorical limits.

In the first three examples below, we consider the \emph{Benjamini-Schramm limit} for random graphs \citep{benjamini01_localLimit,aldous04_localLimit}. 
This limit is local in nature in the sense that it is determined by the behavior of metric balls.

\begin{example}[Increasingly Large Cycles] To illustrate the first source of discrepancy above, consider the family of 2-regular and connected graphs. Let $G_n$ be, deterministically, the cycle graph with $n$ nodes, e.g. $G_8$ is the octagon. Let $G$ be deterministically the linear graph with countably infinitely many nodes. Then $G$ is the Benjamini-Schramm limit of $G_n$ (Cf. Figure 10 of \citep{agostini21_benjaminiSchramm}).
\end{example}

Now, each $G_n$ has (in fact, is) a cycle. Topologically, the 1-dimensional Betti number of each $G_n$ is 1. However, $G$ has no cycle, and the 1-dimensional Betti number is $0$. Therefore, the 1-dimensional Betti number (the topological property) of the Benjamini-Schramm limit (the analytical limit) of $G_n$ is different from the numerical limit (the analytical limit) of the 1-dimensional Betti numbers (the topological property) of $G_n$'s.

The limit fails to capture the loops in the $G_n$'s because \new{the limit is local in nature, while the loops in the $G_n$'s become bigger and bigger and they are eventually not contained in any metric balls of finite radii.}

\begin{example}[Preferential Attachment Graphs] To illustrate the second source of discrepancy above, note that the Benjamini-Schramm limit of preferential attachment graphs is a random tree (Cf. Theorem 1.5 of \citep{garavaglia23_preferentialAttachment_polyaPointTree})
\end{example}

All the algebraic-topological properties of a (random) tree are trivial. \new{In particular, the Betti numbers (the topological property) of the Benjamini-Schramm limit (the analytical limit) of preferential attachment clique complexes, are zero at all positive dimensions almost surely.

On the other hand, our results and those in \citep{siu23_PrefAtt_betti} show that preferential attachment graphs have non-trivial limiting topological properties. By Theorem 3 of \citep{siu23_PrefAtt_betti}, the numerical limits of the expectation (the analytical limit) of the Betti numbers (the topological property) of preferential attachment graphs are infinite, say at dimension 2, for some choices of parameters.



This discrepancy can be explained as follows. The Benjamini-Schramm limit of preferential attachment graphs is a tree because most metric balls in preferential attachment graphs are trees. While these graphs do have nontrivial topological behavior (in terms of Betti numbers), such behavior is supported on a small portion of nodes, namely, the $\Gamma_k$'s in the subsection titled ``Dominating Cycles and Proof Synopsis" of Section 1 of \citep{siu23_PrefAtt_betti}. Therefore, when one takes the Benjamini-Schramm limit by averaging over all nodes, since there are not too many such $\Gamma_k$'s, their effect gets washed out in the limit.
}

\begin{example}[Subtleties of Categorical Limits]
Our homotopy-connectedness result, \cref{thm:infinite_betti}, illustrates the discrepancy between analytical and categorical limits, as well as the subtlety of categorical limits.
\end{example}

First, as the categorical 
colimit 
of the finite complexes, the infinite complex has infinite Betti numbers almost surely for some parameters, whereas the Betti numbers of the Benjamini-Schramm limit, a random tree, vanish almost surely except at dimension 0.

Second, even though the homology groups of the infinite complex are the
colimit 
of the homology groups of the finite complexes, this does \emph{not} mean that the Betti numbers, which are the ranks (or dimensions) of the homology groups, of the \new{infinite complexes are the limits of those of the finite complexes}, as we have noted at the beginning of this subsection. This is because the 
colimit 
takes into account the inclusion maps between different finite complexes, and this information is lost when one takes the numerical limit of Betti numbers.

The noncommutativity and discrepancy are not peculiar to the Benjamini-Schramm limit. Consider, for instance, the Eden model, a model for simulating cell division \citep{eden61_edenModel,auffinger18_firstPassagePercolation}:
\begin{example}[Eden Model]
A cell in $\mathbb{R}^D$ is a volume-1 hypercube whose vertices are in $\mathbb{Z}^D$.
The Eden cluster $(A(t))_{t \in {\mathbb{Z}_+}}$ in $\mathbb{R}^D$ is the stochastic process of subsets of $\mathbb{R}^D$ defined as follows.
\begin{itemize}
\item $A(0)$ is deterministically the cell in the positive octant that contains the origin.
\item $A(t+1)$ is the union of $A(t)$ and a uniformly randomly chosen cell that is outside of $A(t)$ and shares at least a face with at least a cell in $A(t)$.
\end{itemize}
\end{example}
For the Eden model, the limiting cluster is convex \citep{cox81_edenModel_shapeLimit}, and hence has trivial topology,
while the Betti numbers of the cluster diverge to infinity with high probability \citep{manin23_topology_edenModel}. The second source of discrepancy applies in this case, as all nontrivial topological behavior is supported on a thin collar near the boundary of the limiting convex set.

With these examples in mind, one must exercise caution when discussing the limiting behavior of topological properties.

\subsection{Paper Outline}

The rest of the paper is organized as follows. First, we further contextualize our work in the literature in \cref{sec:literature_review}. Then, we state our setup in \cref{sec:setup} and the main results in \cref{sec:main_results}. We prepare for the proofs of the main results by recalling relevant results from the literature in \cref{sec:preliminaries}. We begin proving the main results in \cref{sec:proof_synopsis_intermediate} by giving a synopsis of our proofs, and establishing two intermediate results. We finish the proofs in \cref{sec:proofs}. We collect more technical proofs of intermediate results in \cref{sec:proofs_preliminary_facts}. We discuss our simulations in \cref{sec:simulations} and future directions in \cref{sec:future_directions}. We collect background materials 
in algebraic topology in the appendices. We discuss homology and homotopy theory in \cref{sec:homology_theory,sec:homotopy_theory}. We discuss the difficulty of homotopy theory at the end of \cref{sec:homotopy_theory}, and we contrast the two in \cref{sec:homology_vs_homotopy}.

\subsection{Acknowledgements}
This research was partially funded by the AFOSR grant FA9550-22-1-0091.

This research was partially supported by the Cornell University Center for Advanced Computing, which receives funding from Cornell University, the National Science Foundation, and members of its Partner Program.

The author would like to thank Gennady Samorodnitsky for valuable discussions and feedbacks on the proofs of the main results and on the manuscript; and Shmuel Weinberger for suggesting the question. He also thanks Avhan Misra for assistance with simulations. He thanks Jason Manning, Henry Adams, Lorenzo Ruffoni, and Benjamin Thompson for insightful discussion.

\section{Literature Review}
\label{sec:literature_review}

This work builds on two strands of research, namely the study of higher-order connectivity of scale-free networks, and the study of random simplicial complexes.

Connectivity of scale-free networks has been widely studied. The clustering coefficients of various models were investigated in \citep{bollobas02_scaleFree,holme02_preferentialAttachment,eggemann11_preferentialAttachment_clusteringCoefficient,ostroumova13_preferentialAttachment_clusteringCoefficient,prokhorenkova17_preferentialAttachment_clusteringCoeff}.
In \citep{garavaglia19_subgraphs_preferentialAttachment}, the growth rate of the expected number of small \emph{motifs} in preferential attachment graphs was determined. In \citep{bianconi16_growingSimplicialComplex,courtney17_growingWeightedSimplicialComplex}, some forms of higher-dimensional degrees were considered in some forms of preferential attachment simplicial complexes.

Algebraic-topological properties of scale-free networks are much less studied. The 1-dimensional Betti number of a scale-free simplicial complex was considered in \citep{oh21_bettiNumbers_trianglePreferentialAttachment_bettiNumber}. The asymptotics of the expected Betti numbers of all dimensions of preferential attachment clique complexes were established in \citep{siu23_PrefAtt_betti}. The central limit theorem for the Betti numbers of the age-dependent random connection clique complexes was established in \citep{hirsch23_ADRCM_topology} in the light-tailed regime. To our knowledge, these are the only analytical results about algebraic-topological properties of random scale-free networks in the literature. The current work is largely an extension of \citep{siu23_PrefAtt_betti}. Phase transitions of the Euler characteristic of preferential attachment graphs were numerically observed in \citep{deAmorimFilho22_complexSystem_topologicalPhaseTransition}, where the Euler characteristic is defined in terms of the sum of discrete curvature as opposed to Betti numbers. 

On the other hand, there has been a growing literature on random simplicial complexes and their algebraic-topological properties. We refer the reader to \citep{kahle14_randomSimplicialComplexes_survey,bobrowski22_randomSimplicialComplexes_survey} for comprehensive surveys. Below, we highlight results on homotopy-connectedness and phase transitions. As mentioned in \cref{sec:highlights}, many random simplicial complexes exhibit two phase transitions for each dimension. In some situations, it can also be shown that at some point after the second transition, the complex becomes homotopy-connected. This is the case for the Erdos-Renyi clique complex \citep{kahle14_randomCliqueComplex}, with a homotopy-connectedness result in \citep{kahle09_randomCliqueComplex}; and random Cech and Rips complexes \citep{kahle11_geometricComplex,kahle13_randomComplexes_betti_limitTheorems,yogeshwaran15_geometricComplexes_stationaryPP_betti}. For the Linial-Meshulam models, since they have maximal simplices (simplices that are not faces of other simplices) in only two adjacent dimensions, there is only one phase transition at the lower dimension \citep{linial06_linialMsehulam_2dim,meshulam09_linialMeshulam_dDim,kahle16_linialMeshulam_criticalRegime}. 

For percolation models, phase transition is often defined in terms of the emergence of \emph{giant cycles} as opposed to the surge in the number of cycles. Typically, there is a threshold at which giant cycles of the underlying manifold appear, possibly simultaneously at the threshold \citep{bobrowski20_percolation,duncan23_topologicalPercolation_torus}. Defined slightly differently, the phase transition of the emergence of the giant component in the graph of holes was studied in \citep{hiraoka20_percolation_codim1}.



\section{Setup}
\label{sec:setup}

We first define the preferential attachment model, which was first proposed in \citep{barabasi99_preferentialAttachment}. 
Our notations follow those in \citep{siu23_PrefAtt_betti}.

\begin{definition}[Affine Preferential Attachment Graphs; Definition 4.3.1 of \citep{garavaglia19_preferentialAttachment_thesis}]
\label{def:preferential_attachment}
Let $T$ and $m$ be positive integers with $T \geq 2$ and let $\delta \in (-m, \infty)$. The preferential attachment graph $G(T, \delta, m)$ is the random graph, with no self-loops but possibly with repeated edges, that is constructed inductively as follows.
\begin{itemize}
\item $G(2, \delta, m)$ is the deterministic graph with two nodes, indexed by 1 and 2, and $m$ edges from node 2 to node 1.
\item $G(T, \delta, m)$ is constructed by adding a node, indexed by $T$, to $G(T - 1, \delta, m)$ and $m$ edges from node $T$ to randomly chosen nodes in $G(T - 1, \delta, m)$ one by one from the following conditional distribution: 
\begin{align*}
&P(\text{the } \alpha^{th} \text{ edge from node } T \text{ points to node } v| G(T-1, \delta, m, \alpha-1))
\\=& \frac{1}{C(T, \delta, m, \alpha)}(d_{G(T-1, \delta, m, \alpha-1)}(v) + \delta)
\end{align*}
for $1 \leq \alpha \leq m$,
where $G(T-1, \delta, m, \alpha-1)$ is the graph after adding $\alpha-1$ edges from $T$ to $G(T-1, \delta, m)$, and the normalization constant is $C(T, \delta, m, \alpha) = 2(T-2) m + \alpha - 1 + (T-1) \delta$.
\item $G(\infty, \delta, m) = \cup_{T \geq 2} G(T, \delta, m)$.
\end{itemize}
\end{definition}

\begin{remark}
To the best of our knowledge, the case for $m = 1$ but with a general $\delta$ was first considered in \citep{mori02_preferentialAttachment}.
The definition above is equivalent to the model $PA_T^{(e)}(m, \delta)$ in Definition 4.3.1 of \citep{garavaglia19_preferentialAttachment_thesis} and $PA^{m, \delta}_T(d)$ defined by (1.3.65) of \citep{hofstad24_randomGraphs2}, modulo typos in the latter \citep{hofstad24_randomGraphs2errata}. For lists of other variants of the preferential attachment model, see \citep{prokhorenkova17_preferentialAttachment_clusteringCoeff} and Chapter 4.3 of \citep{garavaglia19_preferentialAttachment_thesis}.
\end{remark}

Throughout this paper,
\begin{equation}\label{eqn:x}
x(\delta, m) = 1 - \frac{1}{2 + \delta/m},
\end{equation}
\new{and our main results will be stated in terms of $x(\delta, m)$.}

\begin{remark} \new{We make two remarks on notations.}
\begin{itemize}
\item Sometimes we may drop the arguments and write $x = x(\delta, m)$ to simplify notation.
\item Despite the convention in the \new{literature of preferential attachment graphs}, we refrain from denoting this quantity by $\chi$ to avoid confusion with the Euler characteristic, a topological invariant.
\end{itemize}
\end{remark}

\new{This quantity $x$ increases with $\delta/m$ and it can be seen as an alternative measure of the strength of preferential attachment. It is related to the rate at which the probability of forming an edge between a new node and an early node converges to $0$. More precisely, we have the following proposition. (See also Exercise 8.13 and Lemma 8.17 of \citep{hofstad16_randomGraphs1} for an analogous theorem about a slightly different preferential attachment model.)

\begin{proposition}\label{prop:interpret_x}
Let $v \leq T$ be positive integers. Let $P(T \to v)$ be the probability that node $T$ is attached to node $v$ in the preferential attachment graph via at least one edge. Then there exist positive constants $c_{m, \delta}, c_{m, \delta, v}, C_{m, \delta}, C_{m, \delta, v}$ such that
$$c_{m, \delta, v} \frac{1}{T^x} = c_{m, \delta} \frac{1}{v^{1-x}T^x} \leq P(T \to v) \leq C_{m, \delta} \frac{1}{v^{1-x}T^x} = C_{m, \delta, v} \frac{1}{T^x}.$$
\end{proposition}

\begin{proof}
We merely sketch the proof because we will not use this proposition in the rest of the paper. For the lower bound, apply Lemma 1 of \citep{garavaglia19_subgraphs_preferentialAttachment} with $\ell = 1$, $\mathbf{u}_\ell = (v)$, $\mathbf{v}_\ell = (T)$ and $\mathbf{j}_\ell = (1)$. For the upper bound, sum over the upper bounds given by the lemma for the same choices of $\ell, \mathbf{u}_\ell$, $\mathbf{v}_\ell$ but all possible choices of $\mathbf{j}_\ell$, namely $(1), ..., (m)$. Note that the lemma has a typo: nodes in $\mathbf{u}_\ell$ should precede those in $\mathbf{v}_\ell$.
\end{proof}

We conclude this section by defining preferential attachment clique complexes. We recall algebraic-topological notions like simplicial complexes in \cref{sec:homology_theory}.
}

\begin{definition}[Preferential Attachment Clique Complexes]
Let $m$ be a positive integer and $\delta \in (-m, \infty)$. For $2 \leq T \leq \infty$, the preferential attachment clique complex $X(T, \delta, m)$ is the simplicial complex with the same nodes as $G(T, \delta, m)$, and $\{v_0, ...,v_q\}$ is a simplex  in $X(T, \delta, m)$ if and only if there is at least one edge between each pair of distinct nodes in $\{v_0, ..., v_q\}$.
\end{definition}

\section{Main Results}
\label{sec:main_results}

Before we state our results, we remark on some probabilistic terminologies and conventions.
\begin{itemize}
\item A property is said to hold \emph{almost surely} if it holds with probability 1.
\item A property is said to hold \emph{asymptotically almost surely} if the probability that it holds converges to 1.
\item Throughout this paper, $C$ and $c$ denote generic constants which may change from line to line or within the same line, unless when they are explicitly stated to be specific constants (the only such exception is in the definition of $B$ right after \cref{eqn:choose_C}). Their subscripts denote variables on which they may depend.
\end{itemize}

Our first main result concerns the scaling limit of the Betti numbers.

\begin{theorem}\label{thm:betti_ass_orderOfMagnitude}
Let $\beta_q(X(T, \delta, m))$ be the Betti number of $X(T, \delta, m)$ at dimension $q$.
Suppose $q \geq 1$, $m \geq 2q$, $x = x(\delta, m) \leq \frac{1}{2q}$, and $\omega(T) \to \infty$. Then the following inequalities hold asymptotically almost surely.\begin{align*}
\frac{1}{\omega(T)} &\leq \frac{\beta_q(X(T, \delta, m))}{T^{1-2qx(\delta, m)}} \leq \omega(T) \qquad \text{ if } q > 1, x < \frac{1}{2q}
\\
\frac{1}{\omega(T)} &\leq \frac{\beta_q(X(T, \delta, m))}{\log T} \leq \omega(T) \qquad \text{ if } q > 1, x = \frac{1}{2q}
\\
-\omega(T) \log T &\leq \beta_q(X(T, \delta, m)) - (m-1)T \leq 1
\qquad \text{ if } q = 1, x < \frac{1}{2q}
\\
-\omega(T) (\log T)^3 &\leq \beta_q(X(T, \delta, m)) - (m-1)T \leq 1
\qquad \text{ if } q = 1, x = \frac{1}{2q}.
\end{align*}
\end{theorem}

Our second main result concerns the homotopy-connectedness of $G(\infty, \delta, m)$:

\begin{theorem}\label{thm:homotopy_connected}
If $q \geq 1$, $m \geq 2(q+1)$ and $x(\delta, m) \leq \frac{1}{2q+2}$, then $X(\infty, \delta, m)$ is $q$-homotopy-connected almost surely.
\end{theorem}

In particular, all Betti numbers of $X(\infty, \delta, m)$ vanish up to dimension $q$, inclusively.

Our third main result addresses the tightness of the condition in the previous theorem.

\begin{theorem}\label{thm:infinite_betti}
Suppose $q \geq 1$ and $m \geq 2q$ and $x = x(\delta, m) > \frac{1}{2q+2}$. Then the $q^{th}$ Betti number of $X(\infty, \delta, m))$, in field coefficients, is almost surely infinite if $x \leq \frac{1}{2q}$; it is almost surely finite if $x > \frac{1}{2q}$ and $q \geq 2$.
\end{theorem}

\begin{remark}
See \cref{def:homology_coefficients} for the definition of homology with coefficients.
\end{remark}

Indeed, if the $q^{th}$ Betti number in rational coefficients is infinite, then by Corollary 3A.6a of \citep{hatcher02_algtopo} (\cref{prop:rationalHomology}), the $q^{th}$ homology group (in integer coefficients) cannot vanish, and hence the space cannot be $q$-homotopy-connected by Hurewicz's theorem (\cref{thm:hurewicz}).
 
\section{Preliminaries}
\label{sec:preliminaries}

In this section, we state the preliminary facts on which the proofs of our main results are based. 
Topological facts are collected in the appendices.

\subsection{An Equivalent Formulation of Finite Preferential Attachment Graphs}

The Polya urn's interpretation of the preferential attachment graph affords an equivalent formulation of the random graph process, which is sometimes more amenable to analytical computations.

\begin{definition}[Finite Polya Preferential Attachment Graphs]\label{def:polya_prefAtt}
Let $m$ be a positive integer and $\delta \in (-m, \infty)$. Let $T \geq 2$ be a positive integer. The finite Polya preferential attachment graph $G_\text{Polya}(T, \delta, m)$ is defined as follows.

\begin{itemize}
\item Let $\psi_1, ..., \psi_T$ be independent random variables with $\psi_1 = 1$ and $\psi_t \sim \text{Beta}(m + \delta, (2m + \delta)t - (3m + \delta))$.
\item Let
\begin{align*}
\varphi_v &= \psi_v \prod_{v < t \leq T} (1-\psi_t)
\\
S_v &= \sum_{i \leq v} \varphi_i
\\&= \prod_{v < t \leq T} (1-\psi_t)
\\
I_v &= [S_{v-1}, S_v).
\end{align*}
\item Let $U_{t, \alpha}$, where $1 \leq t \leq T$ and $1 \leq \alpha \leq m$, be an array of conditionally independent random variables given $\psi_1, ..., \psi_T$, and $U_{t, \alpha} \sim \text{Unif}(0, S_{t-1})$. Let $v_{t, \alpha}$ be the unique index $v$ such that $U_{t, \alpha} \in I_v$.
\item $G_\text{Polya}(T, \delta, m)$ consists of $T$ nodes, indexed by $1, ..., T$, and one edge from node $t$ to node $v_{t, \alpha}$ for each $t$ and $\alpha$.
\end{itemize}
\end{definition}

\begin{theorem}[Theorem 4.4.3 of \citep{garavaglia19_preferentialAttachment_thesis}]\label{thm:polya_equivalent}
If $T$ is finite, then $G(T, \delta, m)$ and $G_\text{Polya}(T, \delta, m)$ have the same distribution.
\end{theorem}

\begin{remark}
The case for $\delta = 0$ is a special case of Theorem 2.1 of \citep{berger14_prefAtt_polyaUrn}. The above theorem is equivalent to Theorem 5.10 of \citep{hofstad24_randomGraphs2} modulo typos.
\end{remark}

Since the two graphs have the same distribution, we drop the subscript ``Polya'' hereafter.

The $\psi_t$'s decouple the dependency of different edges:

\begin{lemma}\label{lem:connection_prob}
Given $\psi_1, ..., \psi_T$, all edges are conditionally independent, and the conditional probability that a node $t$ is connected to a node $v < t$ via its $\alpha^{th}$ edge is
$\psi_v \frac{S_v}{S_{t-1}}$,
which is independent of $\alpha$.
\end{lemma}
\begin{proof} Direct verification, see also (5.3.32) and (5.3.33) of \citep{hofstad24_randomGraphs2}.
\end{proof}

We have the following variant of Proposition 5.18 of \citep{hofstad24_randomGraphs2} and Lemma 3.1 of \citep{berger14_prefAtt_polyaUrn} regarding the asymptotics of $S_v$'s.
\begin{proposition}\label{prop:S_asymptotics}
There exists a positive constant $C_{m,\delta}$ such that for every $0 < \varepsilon < 1$ and, every positive integer $T$, with probability at least $1 - \varepsilon$ it holds that
$$\max_{1 \leq v \leq T} {\lvert \log S_v - x \log (v/T) \rvert} \leq \frac{C_{m,\delta}}{\sqrt{\varepsilon}},$$ and hence there exist positive constants $c_{m, \delta, \varepsilon}, C_{m, \delta, \varepsilon}$ such that for all $1 \leq v \leq T$,
$$c_{m, \delta, \varepsilon} \leq \frac{S_v}{(v/T)^x} \leq C_{m, \delta, \varepsilon}.$$
\end{proposition}

\begin{proof}
Postponed to \cref{sec:proofs_preliminary_facts}.
\end{proof}

\subsection{Homology of Finite Preferential Attachment Clique Complexes}
\label{sec:prelim_finite_pref_att}

We set up for the results from \citep{siu23_PrefAtt_betti} as follows.

\begin{itemize}
\item Let $X$ be a clique complex with vertices $1, ..., T$.
\item Let $L^{(t)}$ be the link of $t$ in $X^{(t)}$. We recall the definition of link in \cref{def:link}.
\item Let $X^{(t)}$ be the subcomplex of $X$ consisting of all simplices of $X$ whose vertices are in $\{1, ..., t\}$.
\item For a subcomplex $S$ of $X$, a nonnegative integer $q$, nodes $s, t$ of $X$ such that $s < t$ and all nodes in $S$ strictly precede $s$, let
\begin{itemize}
\item $\mathcal{S}(S, q, s, t)$ be the event where $S$ is isomorphic to the $(q-1)$-dimensional octahedral sphere $S^{q-1}$ (with the simplicial complex structure of the $\ell_1$-unit sphere in $\mathbb{R}^q$, see the second bullet point of \cref{rmk:simplicial_complex} for a formal definition of $S^{q-1}$) and all nodes in $S$ are connected to both $s$ and $t$, and
\item $b^{(t)}_{IK}(S, s) = \mathbf{1}[(\mathcal{S}(S, q, s, t)] \mathbf{1}[\beta_q(L^{(t)}, S) > 0]$,
\end{itemize}
where $\mathbf{1}[\Lambda]$ denotes the indicator function of the event $\Lambda$. The quantity $\beta_q(L^{(t)}, S)$ is the relative Betti number. Its precise definition, 
which can be found in Section 9 of \citep{munkres84algtopo},
needs not concern us here, as the results below give us all information we need about this quantity.
\end{itemize}

We need the following results.
\begin{proposition}[Proposition 14 of \citep{siu23_PrefAtt_betti}]\label{prop:decomposition}
Let $X$ be a clique complex with vertices labeled by positive integers. Let $q \geq 2$, $S$ be a subcomplex of $X$, and $s$ be a node that is strictly preceded by all nodes in $S$. Then
$$\sum_{s < t \leq T} (\mathbf{1}[\mathcal{S}(S, q, s, t)] - b^{(t)}_{IK}(S, s)) - \sum_{t \leq T} \beta_q(L^{(t)}) \leq \beta_q(X) \leq \sum_{t \leq T} \beta_{q-1}(L^{(t)}).$$
\end{proposition}

\begin{proposition}[Lemmas 19 and 20 of \citep{siu23_PrefAtt_betti}]\label{thm:mean_betti}
Let $q \geq 2$. Suppose $m \geq 2q$. Let $X = X(T, \delta, m)$. Let $S$ be a (possibly random) subcomplex $S$ of $X(T, \delta, m)$, and $s$ be a (possibly random) node in $X(T, \delta, m)$ that is (almost surely) strictly preceded by all nodes in $S$. Then
\begin{align*}
\sum_{t \leq T} E[\beta_{q-1}(L^{(t)})] &\leq C_{\delta, m, q} T^{1-2qx(\delta, m)}\\
E[\sum_{s < t \leq T} b_{IK}^{(t)}(S, s)] &\leq C_{\delta, m, q} T^{1 - (2q+1)x(\delta, m)}
\end{align*} 
whenever the exponents of $T$ are positive. For each inequality, if the exponent is zero or negative, the bound still holds when the expression in $T$ is replaced by $\log T$ or $1$ respectively. 
\end{proposition}

\subsection{A Criterion for Homotopy-Connectedness}

We need the following criterion for homotopy-connectedness.

\begin{theorem}[Barmak; Theorem 4 in the appendix of \citep{Farber23_largeRandomSimpComp}]\label{thm:conic}
Let $q$ be a nonnegative integer and $K$ be a simplicial complex. If every subcomplex $L \leq K$ with at most $2(q+1)$ vertices is contained in the star $\text{St}_K(v)$ of a vertex $v \in K$, then $K$ is $q$-homotopy-connected.
\end{theorem}

\section{Proof Synopsis and Two Intermediate Results}
\label{sec:proof_synopsis_intermediate}

In this section, we sketch our proof and establish two intermediate results (and their corollaries), from which the main results will follow without much complication.

Regarding \cref{thm:betti_ass_orderOfMagnitude} for $q > 1$, since the first moments of the Betti numbers have been estimated in \citep{siu23_PrefAtt_betti}, the upper bounds follow directly from Markov's inequality. 
For the lower bounds, in light of \cref{prop:decomposition}, it suffices to establish a concentration result for the number of common neighbors of a finite set of nodes (namely nodes in $S$). This is the content of \cref{prop:common_neighbors_concentration} below.
Of course, to satisfy the assumptions of \cref{prop:decomposition}, this finite set of nodes has to form a sphere, so we also need a result on the existence of spheres in preferential attachment complexes. This is the content of \cref{prop:induced_sphere_infinite_graph} below.

It turns out the aforementioned concentration result also allows us to apply \cref{thm:conic} to establish \cref{thm:homotopy_connected}.

As in \citep{siu23_PrefAtt_betti}, the proof of \cref{thm:betti_ass_orderOfMagnitude} for $q = 1$ is based on Morse theory.

Finally, given the preliminary facts, \cref{thm:infinite_betti} is a straight-forward corollary of the other results.

We now state the intermediate results we need.

\begin{proposition}\label{prop:common_neighbors_concentration}
Suppose $m \geq k$ and $x(\delta, m) \leq 1/k$.
Fix nodes $v_1 < ... < v_k \leq t_0$ in the infinite preferential attachment graph $G(\infty, \delta, m)$. Let $X_t$ be the indicator of the event where node $t$ is a common neighbor of $v_1, ..., v_k$. Then for every $\varepsilon > 0$ there exists a constant $c_{m, \delta, k, \varepsilon, t_0} > 0$ such that with probability at least $1 - \varepsilon$, for $T \geq T_0(m, \delta, k, \varepsilon, t_0)$,
$$\sum_{t_0 < t \leq T} X_t \geq \begin{cases}
c_{m, \delta, k, \varepsilon, t_0} T^{1-kx(\delta, m)}
& \text{ if } x(\delta, m) < 1/k\\
c_{m, \delta, k, \varepsilon, t_0} \log T & \text{ if } x(\delta, m) = 1/k.
\end{cases}$$
\end{proposition}

We prove this proposition at the end of this section.

\begin{corollary}\label{cor:common_neighbors}
Suppose $m \geq k$. If $x(\delta, m) \leq 1/k$, then every collection of $k$ distinct nodes has infinitely many common neighbors in the infinite preferential attachment graph $G(\infty, \delta, m)$ almost surely.
\end{corollary}

\begin{proof}
The above proposition implies that, for every $\varepsilon > 0$, the probability that the $k$ nodes have at most $N$ common neighbors in $G(\infty, \delta, m)$ is bounded above by $\varepsilon$ by considering $G(T, \delta, m)$ with a sufficiently large $T$. Hence the said probability is in fact $0$. The result follows by taking union with $N$ ranging over all positive integers.
\end{proof}

\begin{remark}
Putting $k = 2$ shows that the diameter of the infinite preferential attachment graph is 2 almost surely. Geodesics between disconnected nodes go through nodes in the distant future. This behavior is markedly different from that of finite preferential attachment graphs, where short paths tend to pass through nodes in a small core of nodes \citep{dommers10_prefAtt_diameter}.
\end{remark}

\begin{proposition}\label{prop:induced_sphere_infinite_graph}
If $m \geq 2q$ and $x(\delta, m) \leq \frac{1}{2q}$, then
$G_\text{simple}(\infty, \delta, m)$ contains an $S^q$ as an induced subgraph almost surely, where $G_\text{simple}$ denotes the graph formed by replacing repeated edges by single edges.
\end{proposition}

\begin{proof}
We prove the proposition by induction. The case for $q = 0$ is trivial, as $S^0$ is just two disconnected points. Inductively, assume $G_\text{simple}(\infty, \delta, m)$ contains an $S^{q-1}$ as an induced subgraph. By \cref{cor:common_neighbors}, the nodes of this $S^{q-1}$ have infinitely many common neighbors, and they cannot all be connected with each other (since the $(m+1)^{st}$ common neighbor cannot be connected with all previous ones). Then $G_\text{simple}(\infty, \delta, m)$ contains, as an induced subgraph, an $S^{k}$ that contains the $S^{q-1}$ above as the equator and two disconnected common neighbors of the nodes of the $S^{q-1}$ as the poles.
\end{proof}

\begin{corollary}\label{cor:induced_sphere_finite_graph}
If $m \geq 2q$ and $x(\delta, m) \leq \frac{1}{2q}$, then for every $\varepsilon > 0$, for $T \geq T_0'(m, \delta, q, \varepsilon)$, with probability at least $1-\varepsilon$,
$G_\text{simple}(T, \delta, m)$ contains an $S^q$ as an induced subgraph.
\end{corollary}

\begin{proof}
For $T \in \{\infty\} \cup \{1, 2, 3, ...\}$, let $\Lambda_T$ be the event that $G_\text{simple}(T, \delta, m)$ contains an $S^q$ as an induced subgraph. Then
$$
\lim_{T \to \infty} P(\Lambda_T) = P(\cup_{T < \infty} \Lambda_T) = P(\Lambda_\infty) = 1.$$
\end{proof}

Before proving \cref{prop:common_neighbors_concentration}, we first establish a lemma.

\begin{lemma}\label{lem:pt_lowerbound}
Consider the Polya urn model (Cf. \cref{def:polya_prefAtt,thm:polya_equivalent}). Suppose $m \geq k$.
Let $p_t$ be the conditional probability that $X_t = 1$ given $\psi_1, ..., \psi_T$. Then for $t \geq v_k$,
\begin{align*}
p_t & \geq 
k! {m \choose k} \prod_{1 \leq i \leq k} \psi_{v_i} \frac{S_{v_i}}{S_{t-1}} - \left(k! {m \choose k}\right)^2 \left(\max_{1 \leq i \leq k} \psi_{v_i} \frac{S_{v_i}}{S_{t-1}}\right) \left(\prod_{1 \leq i \leq k} \psi_{v_i} \frac{S_{v_i}}{S_{t-1}} \right) \notag
\\&\geq
k! {m \choose k} \prod_{1 \leq i \leq k} \psi_{v_i} \frac{S_{v_i}}{S_{t-1}} \left(1 - k! {m \choose k} \frac{\max_{1 \leq i \leq k} S_{v_i}}{S_{t-1}} \right). 
\end{align*}
\end{lemma}

\begin{proof}
The second inequality follows directly from the first one, because $0 \leq \psi_v \leq 1$. For the first inquality, we use the two-term Bonferroni inequality (Exercise 1.6.10 of \citep{durrett19_probability}):
$$P(A_1 \cup ... \cup A_n) \geq \sum_{1 \leq i \leq n} P(A_i) - \sum_{1 \leq i < j \leq n}P(A_i \cap A_j).$$

For distinct integers $\alpha_1, ..., \alpha_k \in [1, m]$, let $A(\alpha_1, ..., \alpha_k)$ be the event where the $t$ is connected to $v_i$ via $t$'s $\alpha_i^{th}$ edge, for $i = 1, ..., k$.

Then $\{X_t = 1\}$ is the union of the $A(\alpha_1, ..., \alpha_k)$'s. There are $k! {m \choose k}$ distinct choices of $(\alpha_1, ..., \alpha_k)$.

By \cref{lem:connection_prob},
\begin{align*}
P(A(\alpha_1, ..., \alpha_k)) &= \prod_{1 \leq i \leq k} \psi_{v_i} \frac{S_{v_i}}{S_{t-1}}.
\\
\intertext{If $(\alpha_1, ..., \alpha_k) \neq (\alpha'_1, ..., \alpha'_k)$, then}
P(A(\alpha_1, ..., \alpha_k) \cap A(\alpha'_1, ..., \alpha'_k))
&= \prod_{1 \leq i \leq k} \left (\psi_{v_i} \frac{S_{v_i}}{S_{t-1}} \right)^{1 + \mathbf{1}[\alpha_i \neq \alpha'_i]}
\\&\leq \left[\max_{1 \leq i \leq k} \left (\psi_{v_i} \frac{S_{v_i}}{S_{t-1}} \right)\right]^{\sum_{1 \leq i \leq k} \mathbf{1}[\alpha_i \neq \alpha'_i]} \prod_{1 \leq i \leq k} \left (\psi_{v_i} \frac{S_{v_i}}{S_{t-1}} \right).
\\&\leq \max_{1 \leq i \leq k} \left (\psi_{v_i} \frac{S_{v_i}}{S_{t-1}} \right) \prod_{1 \leq i \leq k} \left (\psi_{v_i} \frac{S_{v_i}}{S_{t-1}} \right),
\end{align*}
where the last inequality holds because $(\alpha_1, ..., \alpha_k) \neq (\alpha'_1, ..., \alpha'_k)$ implies $\sum_{1 \leq i \leq k} \mathbf{1}[\alpha_i \neq \alpha'_i] \geq 1$, and each $\psi_{v_i} \frac{S_{v_i}}{S_{t-1}}$, being a probability, lies in $[0, 1]$. The lemma then follows by applying the two-term Bonferroni inequality.
\end{proof}

\begin{proof}[Proof of \cref{prop:common_neighbors_concentration}]

Fix $T$. We consider the Polya urn model (Cf. \cref{def:polya_prefAtt,thm:polya_equivalent}) and condition on $\psi_1, ..., \psi_T$.
Then $X_t$'s are conditionally independent Bernoulli random variables.

By \cref{prop:S_asymptotics}, there exist positive constants $c_{m,\delta,\varepsilon},C_{m,\delta,\varepsilon}$ such that the event $\Lambda_1$ defined by
$$c_{m,\delta,\varepsilon} \leq \frac{S_s}{(s/T)^x} \leq C_{m,\delta,\varepsilon} \quad \forall 1 \leq s \leq T$$
has probability at least $1 - \varepsilon/3$.

Let $\Lambda_2$ be the event where $\min_{1 \leq i \leq k} \psi_{v_i} \geq \eta$, where $\eta > 0$ is chosen such that $P(\Lambda_2^C)$ is at most $\varepsilon/3$. This is possible by the continuity of Beta distributions.

Recall that \cref{lem:pt_lowerbound} gives the lower bound of $p_t$, the conditional probability that $X_t = 1$ given $\psi_1, ..., \psi_T$. On $\Lambda = \Lambda_1 \cap \Lambda_2$, which is $(\psi_1, ..., \psi_T)$-measurable, if $t \geq T_1(m, \delta, \varepsilon, t_0)$ for some $T_1(m, \delta, \varepsilon, t_0)$, 
the parenthesized factor (not the binomial coefficients) in the last line of \cref{lem:pt_lowerbound} is at least $1/2$ and hence
$$p_t \geq C_{m, \delta, k, \varepsilon, t_0} \eta^k \prod_{1 \leq i \leq k} \frac{v_i^x}{t^x} \geq C_{m, \delta, k, \varepsilon, t_0} t^{-kx},$$
because $\eta$ depends on $\varepsilon$ and $v_i \geq 1$ for every $1 \leq i \leq n$.

The Chernoff bound then implies that, on $\Lambda$, for every $A, B > 0$,
\begin{align}
P[\sum X_t < B | \psi_1, ..., \psi_T] 
&
\leq \exp [AB - (1 - e^{-A}) \sum p_t] \notag
\\&\leq \exp [AB - (1 - e^{-A}) C_{m, \delta, k, \varepsilon, t_0} \sum t^{-kx}],\label{eqn:choose_C}
\end{align}
where all sums here range over $[T_1(m, \delta, \varepsilon, t_0), T] \cap \mathbb{Z}$.
Picking $A = 1$ and $B$ to be $1/2 \sum t^{-kx}$ times the \emph{specific} $C_{m, \delta, k, \varepsilon, t_0}$ in \cref{eqn:choose_C} gives
$$
P[\sum X_t < B | \psi_1, ..., \psi_T] 
\leq \exp(-C_{m, \delta, k, \varepsilon, t_0} \sum t^{-kx}) \qquad \text{on $\Lambda$}
$$
for some generic positive constant $C_{m, \delta, k, \varepsilon, t_0}$. The divergence of $\sum t^{-kx}$ implies that the right-hand side is less than $\varepsilon/3$ for large enough $T$.

Therefore, by the $(\psi_1, ..., \psi_T)$-measurability of $\Lambda$,
\begin{align*}
&P[\sum X_t < B]
\\&\leq P[\Lambda \cap \{\sum X_t < B\}] + P(\Lambda^C)
\\&= E[\mathbf{1}_\Lambda P(\sum X_t < B | \psi_1, ..., \psi_T)] + P(\Lambda^C) 
\\&\leq \varepsilon/3 + 2\varepsilon/3 = \varepsilon.
\end{align*}
The proposition then follows by noting that $B$ has the desired order of magnitude.
\end{proof}

\section{Proofs of Main Results}
\label{sec:proofs}

\begin{proof}[Proof of \cref{thm:homotopy_connected}]
Apply \cref{cor:common_neighbors} with \cref{thm:conic}. 
\end{proof}

\begin{proof}[Proof of \cref{thm:betti_ass_orderOfMagnitude} for $q > 1$]

We only prove the case for $x < \frac{1}{2q}$. The case for $x = \frac{1}{2q}$ is analogous.

By \cref{thm:mean_betti} and \cref{prop:decomposition}, Markov's inequality 
implies
$$\beta_q(X(T, \delta, m)) \leq \omega(T) T^{1-2qx}$$
asymptotically almost surely.

Similarly, in the notations of \cref{thm:mean_betti} and \cref{prop:decomposition}, it holds asymptotically almost surely that
\begin{align*}
\sum_{t \leq T} \beta_q(L^{(t)}) &\leq T^{1-2qx}\\
\sum_{s < t \leq T} b^{(t)}_{IK}(S, s) &\leq T^{1-2qx}
\end{align*}
regardless of the choice of $S$ and $s$ (as longs as the $b^{(t)}_{IK}(S, s)$'s are well-defined).


Let $S$ be the random subcomplex in $X(T, \delta, m)$ that is the first subcomplex isomorphic to $S^{q-1}$ as an induced subcomplex. If such a subcomplex does not exist, let $S$ be the induced subcomplex on the first $2q$ nodes. (This choice will be immaterial.)

Let $s$ be the random first node after those in $S$ that is connected to all nodes in $S$. If such an $s$ does not exist, let $s$ be the first node after those in $S$.

It suffices to bound 
$$P(S \text{ is not isomorphic to $S^{q-1}$ or $S$ has fewer than $\frac{1}{\omega(T)} T^{1-2qx}$ common neighbors}).$$

Fix $\varepsilon > 0$.
By \cref{cor:induced_sphere_finite_graph}, for $T$ large enough, $P(S \text{ is not isomorphic to $S^{q-1}$}) < \varepsilon$.

By \cref{prop:common_neighbors_concentration}, conditioning on $S$ shows that for $T$ large enough, $$P(\text{$S$ has fewer than $\frac{1}{\omega(T)} T^{1-2qx}$ neighbors}) < \varepsilon.$$

The result then follows.

\end{proof}

\begin{proof}[Proof of \cref{thm:betti_ass_orderOfMagnitude} for $q = 1$]
We follow the proof of Proposition 4 of \citep{siu23_PrefAtt_betti}. Let $V$, $E$, and $F$ be the (multi-)sets of vertices, edges and triangles of $G(T, \delta, m)$. Let $E_\text{simple}$ and $F_\text{simple}$ be the set of edges and triangles of $X(T, \delta, m)$.
The strong Morse inequality (Theorem 1.8 of \citep{forman02_discreteMorse_guide}) implies that
$$- |F_\text{simple}|\leq \beta_1(X(T, \delta, m)) - (|E_\text{simple}| - |V|) \leq \beta_0(X(T, \delta, m)),$$
where $|\cdot|$ denotes cardinality. Let $B$ be the number of ``biangles'' in $G(T, \delta, m)$ (two-node subgraphs with two (repeated) edges between them).
Since $|F_\text{simple}| \leq |F|$ and $|E| - B \leq |E_\text{simple}| \leq |E|$,
$$- (|F| + B) \leq \beta_1(X(T, \delta, m)) - (|E| - |V|) \leq 1,$$

The rest of the proof follows from Markov's inequality 
and Theorem 1 of \citep{garavaglia19_subgraphs_preferentialAttachment}.

\end{proof}

Before we prove \cref{thm:infinite_betti}, we need a lemma to compute the homology groups of the union of a nested sequence of spaces.

\begin{lemma}\label{lem:homology_infinite_complex}
Let $q \geq 0$. Let $X^{(1)}, X^{(2)}, ...$ be a nested sequence of simplicial complexes. Suppose the inclusion maps induce injective homomorphisms on $H_q$ eventually (induced homomorphism defined in \cref{def:induced_homomorphism}). Then
$$H_q(\cup_{t \geq 1} X^{(t)}) \cong \cup_{t \text{ large enough}} H_q(X^{(t)}).$$
\end{lemma}

\begin{proof}
Deferred to the end of \cref{sec:proofs_preliminary_facts}. 
\end{proof}

\begin{proof}[Proof of \cref{thm:infinite_betti}]
In order to invoke \cref{lem:homology_infinite_complex} to reduce the computation to one for finite complexes, we show that the inclusion map induces an injective homomorphism between
$$H_q(X(t-1, \delta, m)) \to H_q(X(t, \delta, m))$$
eventually. Indeed, since $x(\delta, m) > \frac{1}{2q+2}$, for the relevant sum in \cref{thm:mean_betti} to converge to a finite number almost surely, $\beta_q(L^{(t)})$ must be eventually $0$ almost surely, and hence, with field coefficients $H_q(L^{(t)})$ must vanish (Cf. the second bullet point after \cref{def:homology_coefficients}). The segment of the Mayer-Vietoris sequence (\cref{thm:mayer_vietoris}) of $X^{(t)} = X^{(t-1)} \cup \text{St}_{X^{(t)}}(t)$
$$H_q(L^{(t)}) \to H_q(X^{(t-1)}) \to H_q(X^{(t)})$$
then implies that the latter map is eventually injective almost surely.

Now, \cref{lem:homology_infinite_complex} shows it suffices to show $\beta_q(X(T, \delta, m)) \to \infty$ almost surely, or $\beta_q(X(T, \delta, m))$ is bounded almost surely, depending on whether $x \leq \frac{1}{2q}$. For the former case, since $\beta_q(X(T, \delta, m))$ is eventually increasing, it suffices to show that for every $\varepsilon > 0$ and every $B > 0$, with probability at least $1-\varepsilon$,
$\beta_q(X(T, \delta, m)) > B$ eventually. This is a simple consequence of \cref{thm:betti_ass_orderOfMagnitude}. The first case then follows. The other case follows from putting in the first estimate of \cref{thm:mean_betti} into \cref{prop:decomposition} and applying Markov's inequality. 


\end{proof}

\section{Proofs of Preliminary Facts}
\label{sec:proofs_preliminary_facts}

\begin{proof}[Proof of \cref{prop:S_asymptotics}] We largely follow the proof of Lemma 3.1 of \citep{berger14_prefAtt_polyaUrn}. Since $\log S_v = \sum \log (1 - \psi_t)$, consider the approximation
\begin{align*}
&\log S_v - x \log (v/T)
\\&= \left[\log S_v - \sum E \log (1 - \psi_t)\right] + \sum \left[E \log (1 - \psi_t) + E\psi_t \right] - \left[\sum E \psi_t + x\log (v/T)\right],
\end{align*}
where all sums above range over $v < t \leq T$, the same for sums with unspecified ranges hereafter.

The first term, which is the only random term, forms a martingale when $v$ varies (and $T$ is fixed). Kolmogorov's inequality (Theorem 2.5.5 of \citep{durrett19_probability}) 
gives
\begin{align}\label{eqn:random_part_logSv}
P\left(\max_{1 \leq v \leq T} \lvert \log S_v - \sum E \log (1 - \psi_t) \rvert \geq \frac{C}{\sqrt{\varepsilon}} \right) \leq \frac{\varepsilon}{C^2} \sum_{1 \leq t \leq T} \Var \log (1-\psi_t),
\end{align}
where $C$ is to be chosen.
Since $0 \leq - \log(1-a) \leq \frac{a}{1-a}$ for $0 \leq a \leq 1$, for $t \geq 1$,
\begin{align*}
\Var \log (1-\psi_t) \leq E[(\log(1-\psi_t))^2] &\leq E \left [\frac{\psi_t^2}{(1-\psi_t)^2} \right ]
\\&= \frac{B(m + \delta + 2, (2m + \delta)t - (3m + \delta) - 2)}{B(m + \delta, (2m + \delta)t - (3m + \delta))} 
\\&\leq C_{m,\delta}/t^2,
\end{align*}
where $B$ denotes the Beta function. Therefore, the right-hand side of \cref{eqn:random_part_logSv} is bounded from above by $\varepsilon$ if $C^2 = C_{m,\delta} \pi^2/6$.

For the second term, since $-\frac{a^2}{1-a} \leq \log(1-a) + a \leq 0$ for $0 \leq a < 1$,
$$- \sum E[\frac{\psi_t^2}{1 - \psi_t}] \leq \sum E \log (1 - \psi_t) + E \psi_t \leq 0,$$
where, again, the the far-left term is bounded below by $-C_{m,\delta} \pi^2/6$.

For the last term,
$$\sum E\psi_t = \sum \frac{m + \delta}{(2m + \delta)t - (3m + \delta)} = \frac{m + \delta}{2m + \delta} \sum \frac{1}{t - (1 + \frac{1}{3 + \delta/m})}= x \log (T/v) + C_{m,\delta},$$
because $\sum_{1 < v \leq T} \frac{1}{t} = \int_v^T \frac{1}{t} dt + O(1)$.

The result then follows.
\end{proof}

\begin{proof}[Proof of \cref{lem:homology_infinite_complex}]
This is a special case of the theorem in Chapter 14.6 of \citep{may99_algtopo}, which states that under assumptions more general than ours, the homology group of the union is the colimit of the homology groups. In the remainder of the proof, we do the routine checking that the group on the right-hand side is indeed the colimit, which is defined in Chapter 2.6 of \citep{may99_algtopo}.

Denote by $f_{st}: H_q(X^{(s)}) \to H_q(X^{(t)})$ the homomorphism induced by the inclusion $X^{(s)} \subseteq X^{(t)}$.
By assumption, there exists some $t_0$ such that $f_{t,t+1}$ is injective for $t \geq t_0$.
Denote by $A$ the union of groups in the lemma, where ``large enough'' means $t \geq t_0$. Let $j_t: H_q(X^{(t)}) \to A$ be the inclusion homomorphism for $t \geq t_0$. For $t < t_0$, define $j_t = j_{t_0} f_{t, t_0}$.

To check that $A$ is indeed the colimit, it suffices to show that for every abelian group $B$ and every sequence of homomorphisms $g_t: H_q(X^{(t)}) \to B$, if $g_s = g_t f_{st}$ for every $s, t$, then there exists a unique homomorphism $g: A \to B$ such that 
\begin{equation}\label{eqn:colimit_condition}
g_t = g j_t.
\end{equation}

Fix $B$ and the $g_t$'s. Define $g$ as follows. For each $a \in A$, $a = j_t(a_t)$ for some $t \geq t_0$ and $a_t \in H_q(X^{(t)})$. Define $g(a) = g_t(a_t)$, which is equivalent to \cref{eqn:colimit_condition}. Therefore, if $g$ is well-defined, then $g$ is the desired map and it is unique.
To see that $g$ is well-defined, note that once $t$ is chosen, $a_t$ is uniquely determined because $j_t$ is injective for $t \geq t_0$. The choice of $t$ is immaterial because if $a = j_{t'}(a_{t'})$ for some $t' > t$, then $a_{t'} = f_{tt'}(a_t)$ (by the injectivity of $j_{t'}$) and hence $g_{t'}(a_{t'}) = g_{t'}(f_{tt'}(a_t)) = g_t(a_t)$ (by the assumption on the $g_t$'s).

To see that $g$ is a homomorphism, fix $a = j_s(a_s)$ and $b = j_{t}(b_{t})$ with $s, t \geq t_0$. Suppose $s \leq t$. Let $a'_t = f_{st}(a_s)$, hence $a = j_{t}(a'_{t})$. Hence
$$g(a + b) = g(j_{t}(a'_t + b_{t})) = g_t a'_t + g_t b_t = g_t(f_{st}a_{s}) + g(b) = g_s a_s + gb = ga + gb.$$
\end{proof}

\section{Simulations}
\label{sec:simulations}

\begin{figure}[t]
\centering
\includegraphics[height = 4cm]{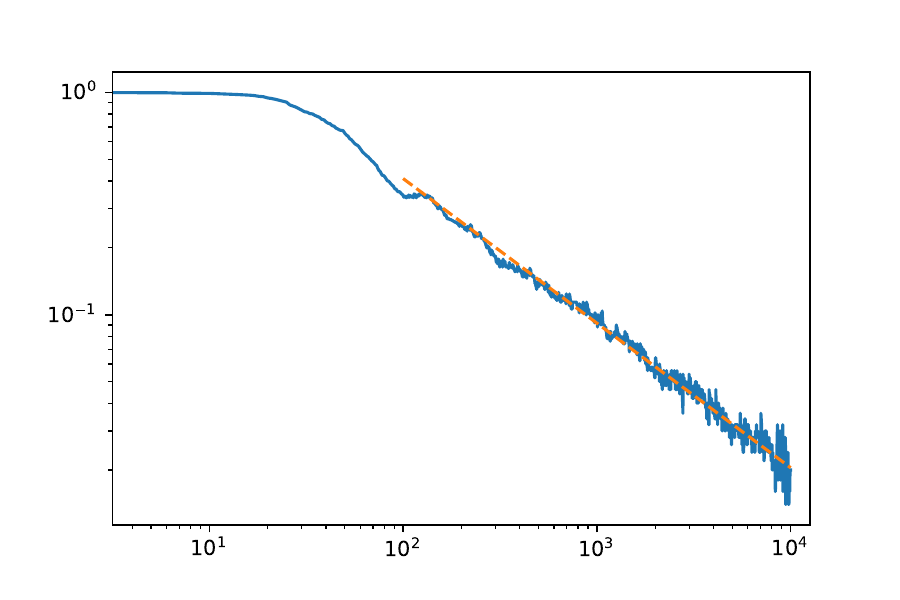}
\includegraphics[height = 4cm]{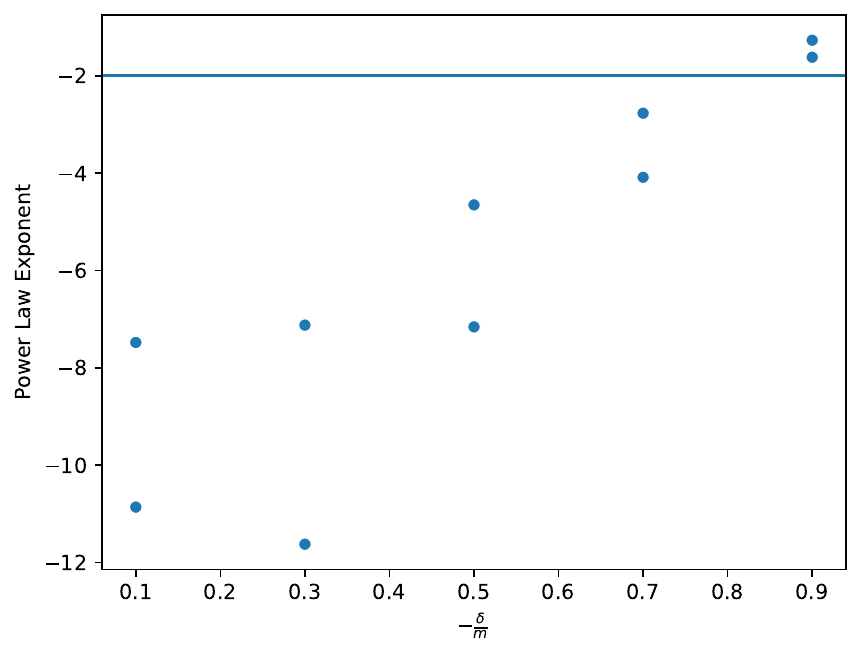}
\caption{(Left) The log-log plot of the evolution of the Kolmogorov-Smirnov norm between the distributions of normalized Betti numbers at time $t$ and time $10000$ (log of the KS norm against $\log t$). The orange dashed line is the least-square line of best-fit. Its slope is $-0.651$.
(Right) The plot of the fitted exponents of the tails of the complementary cumulative distribution functions ($1 - \text{cdf}$) of the Betti numbers of preferential attachment complexes with $10000$ nodes but different choices of $\delta$ and $m$ (exponent against $-\delta/m$). Throughout all simulations, $m$ is either $8$ or $10$. \new{This simulation and the generation of this plot were done by Avhan Misra.}}
\label{fig:simulations}
\end{figure}

In this section, we present further numerical evidence that the mean-normalized Betti numbers of finite preferential attachment clique complexes have limiting distributions that obey power laws. 

We first provide further details about the right panel of \cref{fig:teaser} in the Introduction. We simulated 500 complexes with 10000 nodes and with parameters $m = 7$ and $\delta = -5$ and observe the evolution, as the number of nodes increases, of the distribution of the mean-normalized Betti number $\beta_2/\bar \beta_2$ at dimension 2, where $\bar \beta_2$ is the sample mean of the Betti numbers at dimension 2.

The left panel of \cref{fig:simulations} shows the log-log plot of the evolution of the Kolmogorov-Smirnov norm between the distributions of normalized Betti numbers at time $t$ and time $10000$. It is estimated that the norm converges to 0 at a rate of $t^{-0.651}$.

We now vary the model parameters $\delta$ and $m$ and we increase the number of trials to $1000$. We approximate the limiting distribution by the one for the complex with $10000$ nodes. \new{This simulation was done by Avhan Misra.} The right panel of \cref{fig:simulations}, \new{prepared by Misra}, shows the fitted exponents. We note that the fitted exponents on the far right for $\delta/m = -0.9$ are less negative than $-2$. Therefore, the limiting distribution, if existent, may have an infinite variance. We also note that the right-hand side of the figure also shows that the fitted exponent is not a function of $\delta/m$, even though all our theoretical results depend on $\delta$ and $m$ through $\delta/m$ only.

\section{Future Directions}
\label{sec:future_directions}

Building on \citep{siu23_PrefAtt_betti}, we have established the asymptotic almost sure limits of the orders of magnitude of the Betti numbers of finite preferential attachment clique complexes, and we have established phase transitions for the infinite complexes for each dimension. Open questions about the algebraic-topological properties of preferential attachment complexes abound.

First, numerical evidence suggests the Betti numbers admit a scaling limit. Considerable effort is likely necessary to combine homological algebra and analysis to prove this conjecture. In fact, even the moments of the Betti numbers are not well understood.

Second, as discussed in the Introduction, common notions of analytical limits do not necessarily preserve topological properties. In order to study the large-scale topological behavior of random models, a suitable notion of convergence is yet to be developed.

Third, the topology of other related random models are yet to be studied. Our argument relies heavily on the alternative description of our specific preferential attachment model in \cref{def:polya_prefAtt}. It remains to be seen whether the topological behavior of this model is universal across different modes of preferential attachment. Further, since many real-world networks have different clustering behavior from preferential attachment models \citep{prokhorenkova17_preferentialAttachment_clusteringCoeff}, much work is needed to study the topological behavior of other random models with more complicated clustering behavior.

\appendix



\section{Homology Theory}
\label{sec:homology_theory}

Homology groups and Betti numbers are defined in \cref{def:homology}. Here we fill in gaps in this definition. We follow the exposition in \citep{munkres84algtopo} as far as possible.

First, we define simplicial complexes.

\begin{definition}[Simplicial Complex; Section 3 in Chapter 1 of \citep{munkres84algtopo}]
A simplicial complex $X$ with totally ordered vertices is a collection of finite nonempty subsets of a totally ordered set that is closed under inclusion , i.e. $\tau \in X$ whenever $\emptyset \neq \tau \subseteq \sigma$ for some $\sigma \in X$. Elements in this collection are called simplices, and the dimension of a simplex is one less the number of elements in this simplex (as a subset of the totally ordered set). $0$- and $1$-dimensional simplices are called vertices and edges.
\end{definition}

\begin{remark}
\label{rmk:simplicial_complex}
$\quad$
\begin{itemize}
\item Simplicial complexes thus defined are often called \emph{abstract} simplicial complexes, to distinguish them from their geometric realizations, which are unions of geometric simplices in Euclidean spaces. We only consider abstract simplicial complexes in this paper, with the only exception of the sphere $S^{q-1}$ in the $4^{th}$ bullet point in \cref{sec:prelim_finite_pref_att}. There, we define $S^{q-1}$ as the $\ell_1$ unit sphere, which is a union of geometric simplices in $\mathbb{R}^{q}$. Throughout this paper, we think of $S^{q-1}$ as the abstract simplicial complexes whose geometric realization is such a union. Alternatively, $S^{q-1}$ may be defined as the clique complex (defined in \cref{def:clique_complex}) with vertices $1, ..., 2q$ where $i$ and $j$ are connected by an edge if and only if $i - j \not \equiv 0 \mod q$.
\item Section 3 of \citep{munkres84algtopo} starts with the discussion on oriented simplices. We do not need such notion for \cref{def:homology} because we specialize to the case when the vertices are totally ordered. Indeed, there is a natural orientation on each simplex, namely the monotone ordering.
\end{itemize}
\end{remark}

\begin{definition}[Clique Complex]\label{def:clique_complex} 
A simplicial complex $X$ is said to be a clique complex if for distinct vertices $v_0, ..., v_q$, $\{v_0, ..., v_q\}$ is a simplex whenever there is an edge between every pair of distinct $v_i$ and $v_j$.
\end{definition}

The link of a vertex in a simplicial complex is defined as follows.

\begin{definition}[Star and Link; p.11 of \citep{munkres84algtopo}]\label{def:link} $\quad$
\begin{itemize}
\item The star of a vertex $v$ in $X$ is the subcomplex of $X$ consisting of all simplices containing $v$ (and the faces of these simplices).
\item The link of a vertex $v$ in $X$ is the subcomplex of the star of $v$ that consists of all simplices that do \emph{not} contain $v$.
\end{itemize}
\end{definition}

\begin{remark}
Our notion of star is called the \emph{closed} star in \citep{munkres84algtopo}.
\end{remark}

Next, we address the choice of coefficients.

\begin{definition}[Homology with Arbitrary Coefficients; Section 51 of \citep{munkres84algtopo}]\label{def:homology_coefficients}
Let $X$ be a simplicial complex with totally ordered vertices and $G$ be an abelian group. Let $C_q(X)$ and $\partial_q: C_q(X) \to C_{q-1}(X)$ be the chain groups and boundary maps in \cref{def:homology}. Let $C_q(X; G) = C_q(X) \otimes G$.

The homology group $H_q(X; G)$ at dimension $q$ with coefficients in $G$ is defined by $\ker (\partial_q \otimes \id_G) / \im (\partial_{q+1} \otimes \id_G)$, where $\id_G$ denotes the identity on $G$.
\end{definition}

See Section 10 in Chapter 1 of \citep{munkres84algtopo} for a less intimidating but more verbose definition.

\begin{itemize}
\item If $G = \mathbb{Z}$, one recovers \cref{def:homology}.
\item If $G$ is a field, then the homology groups are vector spaces, and hence they vanish if and only if their Betti numbers vanish.
\item We often denote $H_q(X; G)$ by $H_q(X)$ if the choice of coefficients is implied.
\end{itemize}

\begin{proposition}[Homology with Rational Coefficients; Corollary 3A.6a of \citep{hatcher02_algtopo}]\label{prop:rationalHomology}
Let $X$ be a topological space. Then
$$H_q(X; \mathbb{Q}) \cong H_q(X) \otimes \mathbb{Q}.$$
\end{proposition}

Finally, we define induced maps between homology groups. A simplicial map $f: K \to L$ between two simplicial complexes is a function between the vertex sets of the two complexes such that $\{f(v): v \in \sigma\}$ is a simplex in $L$ for every simplex $\sigma \in K$.

\begin{definition}[Induced Maps Between Homology Groups; Section 2 of Chapter 1 of \citep{munkres84algtopo}]
\label{def:induced_homomorphism}
For every integer $q$, every simplicial map $f: K \to L$ induces a homomorphism $f_\#: C_q(K; G) \to C_q(L; G)$ defined by
$$f_\#\{v_0, ..., v_q\} \otimes g = 
\begin{cases}
(-1)^{\text{sign}(\sigma)}\{f(v_0), ..., f(v_q)\} \otimes g & \text{ if } f(v_0), ..., f(v_q) \text{ are distinct}\\
0 & \text{ otherwise,}
\end{cases}$$
where $\sigma$ is the permutation on $\{0, ..., q\}$ such that $f(v_{\sigma(0)}), ... f(v_{\sigma(q)})$ is monotonically increasing under the ordering of vertices of $L$.
This map in turn induces a homomorphism $f_q: H_q(K; G) \to H_q(L; G)$ defined by
$$f_q(z + \im (\partial^K_{q+1} \otimes \id_G)) = f_\#(z) + \im (\partial^L_{q+1} \otimes \id_G),$$
for every $q$-cycle $z$ of $K$, where the two superscripted $\im \partial_{q+1}$'s are the boundary groups of $K$ and $L$ respectively.
\end{definition}

\begin{theorem}[Mayer-Vietoris Sequence; Theorem 25.1 of \citep{munkres84algtopo}]\label{thm:mayer_vietoris}
Let $X$ and $Y$ be subcomplexes of a simplicial complex $Z$. Then there exist maps such that the following sequence is exact.
$$... \to H_q(X \cap Y) \to H_q(X) \oplus H_q(Y) \to H_q(X \cup Y) \to H_{q-1}(X \cap Y) \to ...$$
\end{theorem}

A sequence $... A_{n+1} \xrightarrow{\varphi_{n+1}} A_n \xrightarrow{\varphi_n} A_{n-1} \to ...$ is said to be exact if $\ker \varphi_n = \im \varphi_{n+1}$ for every $n$. (Cf. Section 23 of \citep{munkres84algtopo}) In particular, if $A_n \cong 0$, then $\varphi_n$ is injective.

\section{Homotopy Theory}
\label{sec:homotopy_theory}

Homotopy and homotopy-connectedness are defined in \cref{def:homotopy_connected}. We start by filling gaps in the definition. We follow the exposition in Chapters 0 and 4 of \citep{hatcher02_algtopo}. For applications of the concept of homotopy-connectedness in graph theory, see Sections 4 and 6 of \citep{bjorner96_topological_combinatorics}.

Homotopy-connectedness is typically defined via homotopy groups, whose elements are, intuitively speaking, spherical ``holes'' of the space.

\begin{definition}[Homotopy Groups; p.340 -- 341 of \citep{hatcher02_algtopo}]
Let $q$ be a nonnegative integer and $S^q$ be the $q$-dimensional sphere. Fix $s_0 \in S^q$.
Let $X$ be a topological space and $x_0 \in X$.

As a set, the homotopy group of $X$ with base-point $x_0$ at dimension $q$ is the set of equivalence classes of maps $f: S^q \to X$ such that $f(s_0) = x_0$, where the equivalence relation is homotopy. In symbols
$$\pi_q(X, x_0) = \{f: f \text{ is a map from $S^q$ to $X$ such that } f(s_0) = x_0\}/\sim,\text{ where}$$
$f \sim g$ if and only if $f$ and $g$ are homotopic.
\end{definition}

\begin{remark}$\quad$
\begin{itemize}
\item Homotopy groups, as their names suggest, are groups, but we will not describe the group operation because we never use it in this paper.
\item Two equivalent definitions of homotopy groups are given in \citep{hatcher02_algtopo}. The version above is stated at the last paragraph of p.340.
\item We note that if $X$ is path-connected, then $\pi_q(X, x_0)$ and $\pi_q(X, x)$ are isomorphic for every $x \in X$.
\end{itemize}
\end{remark}

\begin{example}
For spheres, $\pi_q(S^q) \cong \mathbb{Z}$, and $\pi_r(S^q) \cong 0$ for $r < q$. This is a testimony to the facts that a sphere has a spherical hole of its dimension, and that it has no lower-dimensional hole.
\end{example}

\begin{definition}[Homotopy-Connectedness; p.346 of \citep{hatcher02_algtopo}]
\label{def:homotopy_equivalent_via_group}
A space $X$ is said to be $q$-homotopy-connected if and only if $\pi_r(X, x_0) \cong 0$ for every $0 \leq r \leq q$ and every $x_0 \in X$.
\end{definition}

\begin{remark}$\quad$
\begin{itemize}
\item In \citep{hatcher02_algtopo}, $q$-homotopy-connectedness is abbreviated to $q$-connectedness. We refrain using from this abbreviation to avoid confusion with another graph-theoretic concept, which describes graphs that remain connected after the removal of $q$ vertices.
\item One may readily verify that \cref{def:homotopy_connected} and this definition are equivalent.
\end{itemize}
\end{remark}

Intuitively, in a \emph{1-homotopy-connected} space, also known as a \emph{simply-connected} space, every pair of \emph{paths} are connected by a path of paths; in a \emph{2-homotopy-connected} space, every pair of paths of paths are connected by a path of paths of paths, etc, because the image of the two hemispheres of $S^2$ can be viewed as the pair of paths of paths and the image of the 3-dimensional ball $B^3$ as the path of paths of paths.

\begin{definition}[Contractible; p.4 of \citep{hatcher02_algtopo}]
\label{def:contractible}
A space $X$ is said to be contractible if the identity map on $X$ is homotopic to a constant map.
\end{definition}

\begin{theorem}[Whitehead's Theorem; Theorem 4.5 of \citep{hatcher02_algtopo}]
\label{thm:whitehead}
Let $f: X \to Y$ be a map between two connected CW complexes and let $x_0 \in X$. If $f_*: \pi_q(X, x_0) \to \pi_q(Y, f(x_0))$ is an isomorphism for every positive integer $q$, then $f$ is a homotopy equivalence, i.e. there exists a map $g: Y \to X$ such that $g(f(x_0)) = x_0$, and $gf$ and $fg$ are homotopic to the identity maps on $X$ and $Y$ respectively.
\end{theorem}

\begin{remark} $\quad$
\begin{itemize}
\item Every simplicial complex is a CW complex, which is defined on p.519 of \citep{hatcher02_algtopo}.
\item Homotopy equivalence is defined on p.3 of \citep{hatcher02_algtopo}.
\item For a fixed $q$, $f_*$ is defined as follows. For each map $\sigma: S^q \to X$ such that $\sigma(s_0) = x_0$, $f_*[\sigma] = [f\sigma]$, where $[\cdot]$ denotes the equivalence class under homotopy.
\end{itemize}
\end{remark}

As a corollary, if $Y$ is path-connected and $\pi_q(Y, y_0) \cong 0$ for every positive integer $q$, then letting $f$ be the inclusion of $y_0$ in $Y$ shows $Y$ is contractible. (See the second paragraph of p.348 of \citep{hatcher02_algtopo}.)

\begin{theorem}[Hurewicz's Theorem; Theorem 4.32 of \citep{hatcher02_algtopo}]
\label{thm:hurewicz}
If a space $X$ is $q$-homotopy-connected for some integer $q \geq 1$, then $H_r(X) \cong 0$ for $0 < r \leq q$ and $\pi_q(X, x_0)\cong H_q(X)$ for every $x_0 \in X$.
\end{theorem}

We conclude this section by commenting on the difficulty of homotopy theory and sampling some positive results. The difficulty of homotopy theory is best illustrated by the fact that the computation of the higher homotopy groups of spheres is still an active area of research. In fact, even the computation of the higher stable homotopy groups of spheres, which are better behaved than higher homotopy groups, is difficult.

However, the \emph{ranks} of the higher homotopy groups of spheres are known (Cf. Serre's finiteness theorem, viz. Proposition V.3.3 and Corollary V.6.2 of \citep{serre51_finiteness}): if $r > q$,
$$\rk \pi_r(S^q) =
\begin{cases}
1 & \text{ if }r = 4k-1, q = 2k \text{ for some } k\\
0 & \text{ otherwise.}
\end{cases}
$$
For another positive result regarding the computation of homotopy groups, for each \emph{fixed} $q \geq 2$, there is a polynomial-time algorithm to compute $\pi_q(X)$ for simplicial complexes $X$ such that $\pi_1(X) \cong 0$ \citep{cadek14_homotopyGroupComputation}.

\section{Comparison between Homology and Homotopy Groups}
\label{sec:homology_vs_homotopy}

Homology groups and homotopy groups both capture holes of spaces, but they are subtly different. Homology is more flexible about both holes and trivial holes, in the sense that every map on a sphere gives rise to a cycle, while every path between maps on spheres also gives rise to a boundary. The converse is not true for each case. Consider, for instance, the torus. Its homology group at dimension 2 is $\mathbb{Z}$, which is generated by the torus itself as a cycle. The homotopy group at dimension 2, however, is trivial. In this case, homology captures more structures. On the other hand, consider the torus with a small disc removed. The boundary of the small disc is a nontrivial element in the 1-dimensional homotopy group (it is $aba^{-1}b^{-1}$, where $a$ and $b$ are the longitude and the meridian of the torus), but it is a trivial element in the homology group, because it is the boundary of the whole space as a 2-cycle.



\bibliography{bib_randCpx}

%



\end{document}